\documentclass[a4paper, english]{amsart}
\usepackage{microtype}
\usepackage{amsfonts}

\usepackage{amsmath}
\usepackage{amssymb}
\usepackage[small,nohug,heads=LaTeX]{diagrams}
\usepackage{pstricks}
\usepackage{verbatim}
\usepackage{enumerate}
\usepackage{mathtools}

\newtheorem{MainThm}{Theorem}

\newtheorem{thm}{Theorem}[section]

\newtheorem{prop}[thm]{Proposition}
\theoremstyle{definition}
\newtheorem{defn}[thm]{Definition}
\theoremstyle{remark}
\newtheorem{rem}[thm]{Remark}
\newtheorem{example}[thm]{Example}
\numberwithin{equation}{section}

\newcommand{\bC}{\mathbb{C}}

\newcommand{\bP}{\mathbb{P}}
\newcommand{\bQ}{\mathbb{Q}}
\newcommand{\bR}{\mathbb{R}}

\newcommand{\bZ}{\mathbb{Z}}

\newcommand{\gC}{\bold{C}}

\newcommand{\gM}{\bold{M}}

\newcommand\lra{\longrightarrow}

\newcommand\Diff{\mathrm{Diff}}

\newcommand{\hcoker}{/\!\!/}

\newarrow {Onto} ----{>>}
\newarrow {Equals} ===={=}

\newcommand{\CircNum}[1]{\ooalign{\hfil\raise .00ex\hbox{\scriptsize #1}\hfil\crcr\mathhexbox20D}}

\newcommand{\map}{\mathrm{map}}

\newcommand{\CPinf}{\bC \bP^\infty}
\newcommand{\deq}{:=}

\mathchardef\ordinarycolon\mathcode`\:
\mathcode`\:=\string"8000
\begingroup \catcode`\:=\active
  \gdef:{\mathrel{\mathop\ordinarycolon}}
\endgroup

\title{Relations among tautological classes revisited}

\author{Oscar Randal-Williams}
\thanks{The author was supported by the EPSRC PhD Plus scheme, ERC Advanced Grant No.\ 228082, and the Danish National Research Foundation via the Centre for Symmetry and Deformation.}
\address{Institut for Matematiske Fag\\
Universitetsparken 5\\
DK-2100 K{\o}benhavn {\O}\\
Denmark}
\email{o.randal-williams@math.ku.dk}

\begin{document}

\begin{abstract}
We give a simple generalisation of a theorem of Morita \cite{MoritaJacobianI, MoritaJacobianII}, which leads to a great number of relations among tautological classes on moduli spaces of curves.
\end{abstract}
\maketitle

\section{Introduction}

A \textit{surface bundle} $\Sigma_g \to E \overset{\pi}\to B$ is a smooth fibre bundle with fibre $\Sigma_g$, the oriented surface of genus $g$ without boundary. Such a bundle is equipped with a vertical tangent bundle $T^vE \to E$, which is an oriented 2-plane bundle --- or equivalently a complex line bundle --- over the total space $E$. The Mumford--Morita--Miller classes of this bundle are defined to be
$$\kappa_i(E) \deq \pi_!(c_1(T^vE)^{i+1}) \in H^{2i}(B;\bZ),$$
that is, the pushforward (or fibre integral) along $\pi$ of the $(i+1)$-st power of the first Chern class of the vertical tangent bundle\footnote{The sign convention usual in algebraic geometry would change this by $(-1)^{i+1}$, so that $\kappa_i$ is instead the class obtained by fibre integrating the $(i+1)$-st power of the first Chern class of the vertical \textit{co}tangent bundle.}. Isomorphism classes of such surface bundles are classified by homotopy classes of maps into the \textit{moduli space}
$$\mathcal{M}_g \deq * \hcoker \Diff^+(\Sigma_g) \simeq B\Diff^+(\Sigma_g),$$
and the characteristic classes $\kappa_i$ exist universally as classes $\kappa_i \in H^{2i}(\mathcal{M}_g;\bZ)$. This space is the homotopy type of the moduli stack of Riemann surfaces $\gM_g$, and as we are only discussing homological questions is a good substitute for it.

We define the \textit{tautological algebra} $R^*(\mathcal{M}_g)$ to be the image of the map
$$\bQ[\kappa_1, \kappa_2, ...] \lra H^{2*}(\mathcal{M}_g;\bQ).$$
Our gradings are such that $\kappa_i \in R^i(\mathcal{M}_g) \subset H^{2i}(\mathcal{M}_g;\bQ)$, to conform with the conventions of algebraic geometry.

\vspace{2ex}

Suppose that we wish to parametrise pairs $(\Sigma_g \to E \overset{\pi}\to B, c \in H^2(E;\bZ))$ of a surface bundle and a second integral cohomology class --- or equivalently isomorphism class of complex line bundle --- on the total space. Concordance classes of such pairs are classified by homotopy classes of maps into the moduli space
$$\mathcal{S}_g(\CPinf) \deq \map(\Sigma_g, \CPinf) \hcoker \Diff^+(\Sigma_g)$$
which fits into a fibration sequence
$$\map(\Sigma_g, \CPinf) \lra \mathcal{S}_g(\CPinf) \overset{f}\lra \mathcal{M}_g.$$
A surface bundle $E \overset{\pi}\to B$ determines a homotopy class of maps $B \to \mathcal{M}_g$, and giving a homotopy class of lift of this map up the fibration $f$ is equivalent to giving an element of $H^2(E;\bZ)$.

Given the pair $(\Sigma_g \to E \overset{\pi}\to B, c \in H^2(E;\bZ))$, there are characteristic classes defined by the formula
$$\kappa_{i,j} \deq \pi_!(c_1(T^vE)^{i+1} \cdot c^j) \in H^{2i+2j}(B;\bZ)$$
which generalise the Mumford--Morita--Miller classes, as $\kappa_{i, 0} = \kappa_i$. The fundamental invariant of a class $c \in H^2(E;\bZ)$ is its \textit{degree}, which is defined to be the degree of the class restricted to any fibre --- or equivalently as the integer $\kappa_{-1,1} = \pi_!(c) \in H^0(B;\bZ)$. The moduli space $\mathcal{S}_g(\CPinf)$ splits into path components
$$\mathcal{S}_g(\CPinf) = \coprod_{d \in \bZ} \mathcal{S}_g(\CPinf)_d$$
indexed by the degree. Our main theorem, which is nothing but a simple extension of the work of Morita \cite{MoritaJacobianI, MoritaJacobianII}, gives a characteristic class on $\mathcal{S}_g(\CPinf)$ whose $(g+1)$-st power is rationally trivial. In \S \ref{sec:GeneralisingMorita} we then show how to use this class to produce relations in the cohomology of $\mathcal{M}_g$. As it will occur in almost every formula, we write $\chi = \chi(\Sigma_g) = 2-2g$ for the Euler characteristic of the surface $\Sigma_g$.
\begin{MainThm}\label{thm:Main}
The cohomology class
$$\Omega \deq -\frac{1}{\gcd(\chi,d)^2}\left(\chi^2 \kappa_{-1,2} - 2d\chi \kappa_{0,1} + d^2 \kappa_1 \right) \in H^2(\mathcal{S}_g(\CPinf)_d;\bZ)$$
has the property that $\Omega^{g+1}$ is torsion, annihilated by $\frac{(2g+2)!}{2^{g+1}(g+1)!}$.
\end{MainThm}

In particular, given a surface bundle $\Sigma_g \to E \overset{\pi}\to B$ and a class $c \in H^2(E;\bZ)$ of degree $d$, the pullback of $\Omega$ via the classifying map is
$$\Omega(E,c) \deq -\frac{1}{\gcd(\chi,d)^2} \left ( \chi^2 \pi_!(c^2) - 2d\chi\pi_!(c_1(T^vE)\cdot c) + d^2 \kappa_1(E) \right ) \in H^2(B;\bZ)$$
and the theorem says that $\Omega(E,c)^{g+1}$ is torsion, so in particular trivial in rational cohomology.

\begin{rem}
The construction $c \mapsto \Omega(E,c)$ is easily seen to be non-linear in $c$, so given several classes $c_1$, ..., $c_n \in H^2(E, \bZ)$, we may form the linear combinations $c_A = \sum A_i c_i$ and we will see in \S \ref{sec:GeneralisingMorita} that the collection $\{\Omega(E, c_A)\}_{A \in \bZ^n}$ gives more information than the collection $\{\Omega(E,c_i)\}_{i=1}^n$.
\end{rem}

\subsection{Relation to theorems of Morita}
Morita produced two ways \cite{MoritaJacobianI, MoritaJacobianII} in which one may produce a complex line bundle on the total space of a surface bundle over a space closely related to $\mathcal{M}_g$, though he did not phrase it in this way. Instead, he produced maps from certain surface bundles with section into their associated bundle of Jacobian manifolds $J(\Sigma_g)$, however for pointed surfaces these may be identified with the degree zero Picard varieties $\mathrm{Pic}^0(\Sigma_g)$, and so it is not surprising that this is closely related to studying complex line bundles on the total spaces of surface bundles. We give below a homotopy-theoretic construction of line bundles on certain total spaces of surface bundles, which reproduces the construction of Morita but is more easily generalised.

\vspace{2ex}


We adopt the following useful convention: if $X \to \mathcal{M}_g$ is a space with a map to $\mathcal{M}_g$ classifying a surface bundle, we always write $\pi: \overline{X} \to X$ for the surface bundle it classifies, and $T^v$ for the vertical tangent bundle. We also write $e \deq e(T^v) \in H^2(\overline{X};\bZ)$ for the Euler class of the vertical tangent bundle.

Consider the universal smooth $\Sigma_g$-bundle with section
\begin{equation}\label{eq:UnivBundleSection}
\Sigma_g \lra \overline{\mathcal{M}_g^1} \overset{\pi}\lra \mathcal{M}_g^1
\end{equation}
The section $s: \mathcal{M}_g^1 \to \overline{\mathcal{M}_g^1}$ embeds $\mathcal{M}_g^1$ with a tubular neighbourhood isomorphic to the total space of an oriented 2-dimensional vector bundle, and we let $c := \nu \in H^2(\overline{\mathcal{M}_g^1};\bZ)$ denote the Thom class of this neighbourhood extended to the whole space. This gives a lift of the map $\mathcal{M}_g^1 \to \mathcal{M}_g$ to the space $\mathcal{S}_g(\CPinf)$, and the class $c$ has degree 1, so we obtain a map
$$f:\mathcal{M}_g^1 \lra \mathcal{S}_g(\CPinf)_1$$
classifying this data.

It is easy to see that $\nu^2 = e \cdot \nu$, and so $\kappa_{-1,2} = \kappa_{0,1}$. Furthermore there is a map of bundles
\begin{diagram}
* &\rTo& \mathcal{M}_g^1 &\rTo^{\pi' = \mathrm{Id}}& \mathcal{M}_g^1\\
\dTo & & \dTo^s & & \dTo^{\mathrm{Id}}\\
\Sigma_g & \rTo & \overline{\mathcal{M}_g^1} & \rTo^\pi & \mathcal{M}_g^1
\end{diagram}
and as $\nu$ is the Thom class of the normal bundle to $s$, we have the equation $\pi_!(X \cdot \nu) = \pi'_!(s^*(X)) = s^*(X)$, and so $\pi_!(\nu \cdot e) = e \in H^2(\mathcal{M}_g^1;\bZ)$. Then the formula of Theorem \ref{thm:Main} is
$$\Omega = (2-\chi)\chi e - \kappa_1 \in H^2(\mathcal{M}_g^1;\bZ)$$
and hence $((2-\chi)\chi e - \kappa_1)^{g+1}$ is trivial in the group $H^*(\mathcal{M}_g^1;\bQ)$, which appears in Morita's work as \cite[Theorem 1.6]{MoritaJacobianII}. Fibre integrating this equation to $\mathcal{M}_g$ gives the relation
\begin{equation*}
\sum_{i=0}^g \binom{g+1}{i+1} \left ( \frac{\kappa_1}{\chi(\chi-2)} \right )^{g-i} \cdot \kappa_i = 0 \in H^{2g}(\mathcal{M}_g;\bQ),
\end{equation*}
and multiplying $\Omega^{g+1}$ by $e^k$ for $k > 0$ and then fibre integrating gives the relation
\begin{equation}\label{eq:Morita1}
\sum_{i=-1}^{g} \binom{g+1}{i+1} \left ( \frac{\kappa_1}{\chi(\chi-2)} \right )^{g-i}\cdot \kappa_{i+k} = 0 \in H^{2g+2k}(\mathcal{M}_g;\bQ),
\end{equation}
in both cases bearing in mind that $\kappa_0 = \chi$.

\vspace{2ex}

We now give the second construction, and explain why it gives a homotopical version of the fibrewise Jacobi map as described by Morita \cite{MoritaJacobianI}. Let us write $\mathcal{M}_g(2)$ for the space $\Sigma_g \times \Sigma_g \hcoker \Diff^+(\Sigma_g)$, which parametrises surface bundles with two ordered (not necessarily distinct) points in each fibre. We write $\overline{\mathcal{M}_g(2)}$ for its universal bundle, so there is a smooth $\Sigma_g$-bundle
$$\Sigma_g \lra \overline{\mathcal{M}_g(2)} \overset{\pi}\lra \mathcal{M}_g(2)$$
with two sections, $s_1$ and $s_2$, which each embed $\mathcal{M}_g(2)$ with a tubular neighbourhood isomorphic to the total space of an oriented 2-dimensional vector bundle. As in the previous construction they each produce a class $\nu_1, \nu_2 \in H^2(\overline{\mathcal{M}_g(2)};\bZ)$, and we may let $c := \nu_1 - \nu_2$. This has degree zero on each fibre, and hence there is a classifying map $f: \mathcal{M}_g(2) \to \mathcal{S}_g(\CPinf)_0$ for this data.

The formula of Theorem \ref{thm:Main} is then
$$\Omega = -\pi_!((\nu_1-\nu_2)^2) = (2\nu_{12} - e_1 - e_2) \in H^2(\mathcal{M}_g(2);\bZ)$$
where $\nu_{12}$ is the cohomology class Poincar\'{e} dual to the locus of surfaces with two marked points where the two marked points coincide, and $e_i$ is the Euler class of the tangent bundle at the $i$-th marked point. This appears in Morita's work as \cite[Theorem 2.1]{MoritaJacobianI}, and produces relations in the cohomology of $\mathcal{M}_g(2)$ because $(2\nu_{12} - e_1 - e_2)^{g+1}=0$, but also relations in $\mathcal{M}_g(1)$ and $\mathcal{M}_g$ by fibre integration. As there are easy identities $\nu_{12}^2 = \nu_{12}e_1 = \nu_{12} e_2$, the class $(2\nu_{12} - e_1 - e_2)^{g+1}$ may be expanded and fibre integrated to see that for each $h \geq 0$, $\kappa_{g+h-1}$ is expressible in terms of sums of products of $\kappa$-classes of lower degree, which reproves Mumford's theorem \cite[\S 5--6]{Mumford} that the classes $\kappa_1, ..., \kappa_{g-2}$ generate the tautological algebra.

Let us now relate this construction to Morita's work \cite{MoritaJacobianI} on the fibrewise Jacobi map. The universal surface bundle over $\mathcal{M}_g$ has an associated bundle of first homology groups, and we may take the fibrewise classifying space to obtain a bundle
\begin{equation}\label{eq:JacobianBundle}
BH_1(\Sigma_g;\bZ) \lra \widetilde{\mathcal{M}}_g \lra \mathcal{M}_g.
\end{equation}
Just as the fundamental group of $\mathcal{M}_g$ is the \emph{mapping class group} $\Gamma_g$, the fundamental group of $\widetilde{\mathcal{M}}_g$ is the \emph{extended mapping class group} $\widetilde{\Gamma}_g$ introduced by Kawazumi \cite{Kawazumi98}. Recall that for a Riemann surface $\Sigma$, its \emph{Jacobian} is the torus $H_1(\Sigma;\bR)/H_1(\Sigma;\bZ)$, and as $H_1(\Sigma;\bR)$ is a contractible free $H_1(\Sigma;\bZ)$-space this is a model for $BH_1(\Sigma;\bZ)$. It is clear that (\ref{eq:JacobianBundle}) is fibre homotopy equivalent to the bundle of fibrewise Jacobian varieties.

Note that $\bC\bP^\infty$ is the subgroup of constant maps the topological abelian group $\mathrm{map}^0(\Sigma_g, \bC\bP^\infty)$, and is preserved by the action of $\Diff^+(\Sigma_g)$. Hence $\Diff^+(\Sigma_g)$ acts on the quotient $\mathrm{map}^0(\Sigma_g, \bC\bP^\infty) / \bC\bP^\infty$, which one easily sees to be equivariantly homotopy equivalent to $BH_1(\Sigma_g;\bZ)$. Hence there is a diagram
\begin{diagram}
\bC\bP^\infty &\rTo & \mathrm{map}^0(\Sigma_g, \bC\bP^\infty) &\rTo & BH_1(\Sigma_g;\bZ)\\
\dEq & & \dTo & & \dTo\\
\bC\bP^\infty &\rTo & \mathcal{S}_g(\bC\bP^\infty)_0 & \rTo & \widetilde{\mathcal{M}}_g\\
& & \dTo & & \dTo\\
& & \mathcal{M}_g & \rEq & \mathcal{M}_g
\end{diagram}
where the rows and columns are fibrations, and the top right hand square is a homotopy pullback.

We thus have the composition
$$\mathcal{M}_g(2) \overset{f}\lra \mathcal{S}_g(\bC\bP^\infty)_0 \lra \widetilde{\mathcal{M}}_g$$
of maps over $\mathcal{M}_g$. On fibres this gives
$$\Sigma_g \times \Sigma_g \lra \mathrm{\map}^0(\Sigma_g, \bC\bP^\infty) \lra BH_1(\Sigma_g;\bZ)$$
where the first map sends a pair $(x,y)$ to the Poincar\'{e} dual of $x-y$, and the second map takes the quotient by $\bC\bP^\infty$. Hence on each factor the map is the identity map on first homology.

Observing that $\mathcal{M}_g(2) = \overline{\mathcal{M}_g(1)}$, the composition can also be viewed as the middle row in a map of fibrations
\begin{diagram}
\Sigma_g & \rTo & BH_1(\Sigma_g)\\
\dTo & & \dTo \\
\overline{\mathcal{M}_g(1)} & \rTo &\widetilde{\mathcal{M}}_g\\
\dTo & & \dTo \\
\mathcal{M}_g(1) & \rTo & \mathcal{M}_g,
\end{diagram}
which is given by a map $\Sigma_g \to BH_1(\Sigma_g)$ on each fibre with the property that it induces the identity map in first homology: this is a homotopical description of the fibrewise Jacobi map. 


\subsection{Looijenga's theorem}
Looijenga \cite{Looijenga} has shown that the tautological algebra is trivial above degree $(g-2)$, and has dimension at most 1 in degree $(g-2)$. Using the Witten conjecture, it can be shown \cite[Theorem 2]{Faber} that $\kappa_{g-2} \neq 0$ on $\mathcal{M}_g$, and so $R^{g-2}(\mathcal{M}_g)$ has precisely dimension 1. Unfortunately, all the relations in $R^*(\mathcal{M}_g)$ obtained in the previous section (by fibre integrating relations on $\mathcal{M}_g(1)$ and $\mathcal{M}_g(2)$ down to $\mathcal{M}_g$) lie in degrees greater than $ g-2$, and so in the light of Looijenga's theorem are not so interesting.

Given a monomial $\kappa_1^{i_1} \cdots \kappa_{g-2}^{i_{g-2}}$ such that $\sum_{j=1}^{g-2} j \cdot i_j = g-2$, one can ask for the unique rational numbers $M(g;i_1, ..., i_{g-2})$ such that
$$\kappa_1^{i_1} \cdots \kappa_{g-2}^{i_{g-2}} = M(g;i_1, ..., i_{g-2}) \cdot \kappa_{g-2}.$$
This means finding relations in $R^{g-2}(\mathcal{M}_g)$, and in the following section we give a procedure for finding relations in this and lower degrees.

\subsection{Generalising Morita's theorems}\label{sec:GeneralisingMorita}
Let us write $\mathcal{M}_g(n) := (\Sigma_g)^n \hcoker \Diff^+(\Sigma_g)$ for the moduli space of genus $g$ surfaces with $n$ ordered, not necessarily distinct points on them. Note that $\mathcal{M}_g(1) = \mathcal{M}_g^1$. Consider the universal family
\begin{equation}\label{eq:UnivBundleN}
\Sigma_g \lra \overline{\mathcal{M}_g(n)} \overset{\pi}\lra \mathcal{M}_g(n)
\end{equation}
which has $n$ sections $s_1$, ..., $s_n$. Each section $s_i: \mathcal{M}_g(n) \to \overline{\mathcal{M}_g(n)}$ embeds $\mathcal{M}_g(n)$ with a tubular neighbourhood isomorphic to the total space of an oriented 2-dimensional vector bundle, and we let $\nu_i \in H^2(\overline{\mathcal{M}_g(n)};\bZ)$ denote the Thom class of this neighbourhood extended to the whole space. For each vector $A = (A_1, ..., A_n) \in \bZ^n$, define
$$c_A \deq \sum_{i=1}^n A_i \nu_i.$$
This gives a lift of the map $p : \mathcal{M}_g(n) \to \mathcal{M}_g$ to the space $\mathcal{S}_g(\CPinf)$, and the class $c_A$ has degree $d_A := \sum A_i$, so we obtain a map
$$f:\mathcal{M}_g(n) \lra \mathcal{S}_g(\CPinf)_{d_A}$$
classifying this data. Then it is easy to calculate that
$$\Omega_A = \sum_{i=1}^n (\chi^2 A_i^2 - 2d_A \chi A_i) e_i + 2\chi^2 \sum_{i<j} A_i A_j\nu_{ij} + d_A^2 \kappa_1 \in H^2(\mathcal{M}_g(n);\bZ)$$
and so $\Omega_A^{g+1} = 0$ in rational cohomology \emph{for all vectors} $A \in \bZ^n$. In particular, considered as a polynomial in the variables $A_1$, ..., $A_n$ with coefficients in $H^{2g+2}(\mathcal{M}_g(n);\bQ)$, the expression $\Omega_A^{g+1}$ is the zero polynomial. This gives a tremendous number of trivial tautological classes in $H^{2g+2}(\mathcal{M}_g(n);\bQ)$, which may be fibre integrated to $\mathcal{M}_g$ to give trivial tautological classes in the group $H^{2g+2-2n}(\mathcal{M}_g;\bQ)$.

\begin{example}
For $n=2$ we have
$$\Omega_A = A_1^2(\chi(\chi-2)e_1+\kappa_1) + 2A_1A_2(\chi^2 \nu_{12}-\chi(e_1+e_2)+\kappa_1) + A_2^2(\chi(\chi-2)e_2+\kappa_1)$$
and one may calculate the coefficient $X_i$ of $A_1^iA_2^{2g+2-i}$ in $\Omega_A^{g+1}$ to be
$$\sum_{a=0}^{\lfloor i/2 \rfloor}(\chi(\chi-2)e_1+\kappa_1)^a(2\chi^2 \nu_{12} -2\chi(e_1+e_2)+2\kappa_1)^{i-2a}(\chi(\chi-2)e_2+\kappa_1)^{g+1-i+a}$$
which must thus be trivial in $H^{2g+2}(\mathcal{M}_g(2);\bQ)$. 
\end{example}

\subsection{Fibre-integration and graphs}
In order to use the above relations to obtain relations on $\mathcal{M}_g$ we require a way to compute the fibre integral of a monomial
$$e^a \nu^b \kappa_1^c := (e_1^{a_1} \cdots e_n^{a_n}) \cdot( \nu_{12}^{b_{12}} \nu_{13}^{b_{13}} \cdots \nu_{n-1,n}^{b_{n-1,n}}) \cdot \kappa_1^c$$
along the map $\pi : \mathcal{M}_g(n) \to \mathcal{M}_g$.

\begin{defn}
A \emph{weighted graph} $\Gamma$ on the set $\{1, 2, ..., n\}$ is a graph on this set of vertices with each vertex labelled by a natural number, and each edge labelled by a non-zero natural number. We denote by $w(\Gamma)$ the total weight of $\Gamma$, and by $v(\Gamma)$ and $e(\Gamma)$ the total number of vertices and edges of $\Gamma$ respectively, counted without weight.
\end{defn}

To each monomial $e^a \nu^b \kappa_1^c$ we associate a weighted graph $\Gamma(a, b)$ by weighting the vertex $\{i\}$ by $a_i$, and putting an edge between $\{i\}$ and $\{j\}$ if $b_{ij} \neq 0$, and weighting it $b_{ij}$. Using the standard relations $e_i \nu_{ij} = e_j \nu_{ij}$ and $\nu_{ij} \nu_{jk} = \nu_{ij} \nu_{ik}$, we see that the monomials $e^a \nu^b \kappa_1^c$ and $e^{a'} \nu^{b'} \kappa_1^{c'}$ are equal precisely when the graphs $\Gamma(a,b)$ and $\Gamma(a',b')$ differ by a sequence of the following moves:
\begin{enumerate}[(i)]
\item Moving weight from an edge to an adjacent vertex, and vice-versa.
\item Given two edges leaving the same vertex $x$ and going to $y$ and $z$ respectively, where there is no edge $yz$, we may create an edge $yz$ of weight 1 at the expense of removing 1 weight from $xy$ or $xz$.
\end{enumerate}

By reducing weighted graphs to weighted trees, it is then easy to see that
$$\pi_!(e^a \nu^b \kappa_1^c) = \kappa_1^c \prod_{\Gamma' \subset \Gamma(a,b)} \kappa_{w(\Gamma') - v(\Gamma')}$$
where the product is taken over the connected components of $\Gamma(a,b)$. In this formula we must remember that $\kappa_i = 0$ for all $i < 0$, and $\kappa_0 = \chi$

\begin{example}
Let us look again at the case $n=2$, where we have
$$\Omega_A = (\chi^2 A_1^2 - 2d_A\chi A_1) e_1 + (\chi^2 A_2^2 - 2d_A\chi A_2) e_2 + 2\chi^2 A_1A_2 \nu_{12} + (d_A)^2 \kappa_1.$$
The above description in terms of weighted graphs says that the coefficient of $\kappa_{g-1}$ in $\pi_!(\Omega_A^{g+1})$ is the coefficient of $t^{g+1}$ in
$$\sum_{a_1, a_2 = 0}^\infty \sum_{b=1}^\infty (\chi^2 A_1^2 - 2d_A\chi A_1)^{a_1} (\chi^2 A_2^2 - 2d_A\chi A_2)^{a_2} (2A_1A_2)^b t^{a_1 + a_2 + b}.$$
Note that the coefficients are symmetric functions in the $A_i$, so can be written in terms of $\sigma_1 = A_1 + A_2$ and $\sigma_2 = A_1 A_2$. After some work, one can see that this formal power series is
$$ \left ( \frac{1}{1 - 2(\chi(\chi-2)\sigma_1^2 - 2\chi^2\sigma_2)t + (\chi^4\sigma_2^2 + (4\chi^2-2\chi^3)\sigma_1^2\sigma_2)t^2} \right ) \left ( \frac{2\sigma_2 t}{1-2\sigma_2 t} \right)$$
and by specialising to $\sigma_1=0$ one may compute that all the coefficients of $t^{g+1}$ are non-zero for $g \geq 2$. Hence we see that the coefficient of $\kappa_{g-1}$ is non-zero in the relation, and hence it is decomposable on $\mathcal{M}_g$.
\end{example}

\subsection{Calculations}
For general $n$, it is a formidable problem of combinatorial algebra to extract explicit relations from the above description, but we give in the following examples some computer calculations we have made in low genus, which demonstrate that the relations obtained are non-trivial. The calculations were made in {\sc Magma} \cite{Magma}.

\begin{example}
Doing the above for $g=3$ and $n=2, 3$ yields the relations
\begin{eqnarray*}
\kappa_1^2 = \kappa_2 = 0
\end{eqnarray*}
in $H^*(\mathcal{M}_3;\bQ)$.
\end{example}

\begin{example}
Doing the above for $g=4$ and $n=2, 3$ yields the relations
\begin{eqnarray*}
\kappa_2^2 = \kappa_1\kappa_2 = \kappa_3 &=& 0\\
    3\kappa_1^2 &=& 32\kappa_2
\end{eqnarray*}
in $H^*(\mathcal{M}_4;\bQ)$.

\end{example}

\begin{example}
Doing the above for $g=5$ and $n=2, 3$ yields the relations
\begin{eqnarray*}
\kappa_4 = \kappa_1 \kappa_3 = \kappa_2^2 = \kappa_2 \kappa_3 = \kappa_3^2 &=& 0\\
\kappa_1^3 &=& 288\kappa_3\\
\kappa_1\kappa_2 &=& 20\kappa_3
\end{eqnarray*}
in $H^*(\mathcal{M}_5;\bQ)$. There is a relation in degree 4 (namely $5\kappa_1^2 = 72\kappa_2$) which we have not detected, as for $g=5$ and $n=3$ we only obtain relations in degree 6 (and above). Possibly using the same method for $n=4$ we may obtain this relation, although the problem becomes computationally unwieldy very quickly.
\end{example}

\begin{example}
Doing the above for $g=6$ and $n=2,3$ yields the relations
\begin{eqnarray*}
\kappa_4^2 = \kappa_3\kappa_4 = \kappa_2\kappa_4= \kappa_2\kappa_3= \kappa_1\kappa_4= \kappa_5 &=& 0\\
    5 \kappa_1^4 &=& 73728 \kappa_4 \\
    5 \kappa_1^2 \kappa_2 &=& 4064 \kappa_4\\
    5 \kappa_2^2 &=& 226 \kappa_4\\
    \kappa_1 \kappa_3 &=& 32 \kappa_4
\end{eqnarray*}
in $H^*(\mathcal{M}_6;\bQ)$. There are also relations in degree 6 which we are not able to see using $n=3$, but could perhaps find using $n=4$.
\end{example}

Finally we give an example of the relations obtained in higher genus. Although the $n=2$ calculations at genus 9 take only a few seconds, the $n=3$ calculations take about fifteen minutes. We have also made calculations at genus 12, where the $n=3$ calculations take several hours, but these give a very incomplete set of relations, and we do not include them here. We suspect that a less na\"{i}ve algorithm would be able to go to significantly higher degrees.

\begin{example}
Doing the above for $g=9$ and $n=2,3$ yields the relations
\begin{eqnarray*}
    \kappa_7^2 = \kappa_6\kappa_7 = \kappa_5\kappa_7 = \kappa_4\kappa_7 = \kappa_3\kappa_7 &=& 0\\
    \kappa_2\kappa_7 = \kappa_4^2 =\kappa_3\kappa_5 = \kappa_2\kappa_6 = \kappa_1\kappa_7 = \kappa_8 &=& 0\\
    \kappa_1^7 &=& 26011238400\kappa_7 \\
    7\kappa_1^5 \kappa_2 &=& 6195953664\kappa_7\\
    35\kappa_1^3\kappa_2^2 &=& 1060508672\kappa_7\\
    35\kappa_1\kappa_2^3 &=& 36513024\kappa_7\\
    35\kappa_1^4\kappa_3 &=& 743491584\kappa_7\\
    5\kappa_1^2\kappa_2\kappa_3 &=& 3667456\kappa_7\\
    35\kappa_2^2\kappa_3 &=& 891328\kappa_7\\
    7\kappa_1\kappa_3^2 &=& 125824\kappa_7\\
    7\kappa_1^3\kappa_4 &=& 2796544\kappa_7\\
    7\kappa_1\kappa_2\kappa_4 &=& 97536\kappa_7\\
    7\kappa_3\kappa_4 &=& 2424\kappa_7\\
    \kappa_1^2\kappa_5 &=& 6240\kappa_7\\
    \kappa_2\kappa_5 &=& 220\kappa_7\\
    \kappa_1\kappa_6 &=& 84\kappa_7.
\end{eqnarray*}
Note that $\bQ[\kappa_1, ..., \kappa_7]$ has rank 15 (the number of partitions of 7) in degree 14, but by Looijenga's theorem $R^*(\mathcal{M}_9)$ has rank one in degree 14. We see above precisely 14 relations, so have determined all intersection numbers on $\mathcal{M}_9$.
\end{example}

In our calculations at $g=12$, $n=3$, we do not find all the intersection numbers on $\mathcal{M}_{12}$. We find 34 linearly-independent relations, but $\bQ[\kappa_1, ..., \kappa_{10}]$ has rank 77 in degree 20 and so we should find 76 relations.

\vspace{2ex}

In these examples, even in high degrees where all tautological classes become trivial, each relations cannot be obtained from fibre integrating any single trivial tautological class on $\mathcal{M}_g(n)$ coming from the above construction. Rather, each individual relation typically gives a linear equation among the tautological classes in a certain degree with very large coefficients, and this system of equations turns out to be highly overdetermined.

\subsection{Relation to the work of Faber}

Faber \cite{Faber} has made detailed conjectures concerning the algebraic structure of the tautological algebra $R^*(\mathcal{M}_g)$, and has verified these conjectures at least up to genus 23 by producing sufficiently many relations among the $\kappa_i$.

The method used by Faber to produce relations is similar to that used here in that it also involves producing relations on $\mathcal{M}_g(n)$ (or rather its algebraic analogue, $\gC_g^n$, the iterated fibre product of the universal curve) and pushing them forward to $\mathcal{M}_g$. However, Faber's method for producing relations on $\mathcal{M}_g(n)$ is completely algebro-geometric and does not seem to be related to Morita's. It seems that even the relations that are obtained on $\mathcal{M}_g(n)$ are different: ours involve only the classes $\kappa_1$, $e_i$ and $\nu_{ij}$, and the higher $\kappa$ classes appear only after fibre integration, whereas Faber's seem to already include the higher $\kappa$ classes.

\section{Proof of Theorem \ref{thm:Main}}
Let us first treat the degree zero case, where we must show that
$$\Omega \deq -\kappa_{-1,2} \in H^2(\mathcal{S}_g(\CPinf)_0;\bZ)$$
has the property that $\Omega^{g+1}$ is torsion annihilated by $\frac{(2g+2)!}{2^{g+1}(g+1)!}$. Recall from \cite{ERW10} that there is a fibration
\begin{equation}\label{eq:TopologicalJacobianFibration}
\CPinf \lra \mathcal{S}_g(\CPinf)_0 \overset{\pi}\lra B\widetilde{\Gamma}_g
\end{equation}
where the base space is the classifying space of the \textit{extended mapping class group}
$$\widetilde{\Gamma}_g \deq H_1(\Sigma_g;\bZ) \rtimes \Gamma_g$$
which fits into a fibration
$$BH_1(\Sigma_g;\bZ) \lra B\widetilde{\Gamma}_g \lra B\Gamma_g.$$
Morita has shown \cite{MoritaJacobianI} that there is a class $\Omega \in H^2(B\widetilde{\Gamma}_g;\bZ)$ enjoying the properties
\begin{enumerate}[(i)]
    \item $\Omega$ restricts to twice the symplectic form in $H^2(BH_1(\Sigma_g;\bZ);\bZ) \cong \wedge^2 H_1(\Sigma_g;\bZ)$,
	\item $s^*(\Omega) =0 \in H^2(B\Gamma_g;\bZ)$ where $s$ is the canonical section,
	\item $\Omega^{g+1} \in H^{g+2}(B\widetilde{\Gamma}_g;\bZ)$ is torsion annihilated by $\frac{(2g+2)!}{2^{g+1}(g+1)!}$.
\end{enumerate}
Thus the degree zero case of Theorem \ref{thm:Main} will follow one we show that this $\Omega$ pulls back to $-\kappa_{-1,2}$ on $\mathcal{S}_g(\CPinf)_0$.

\begin{prop}
The class $\pi^*(\Omega) \in H^2(\mathcal{S}_g(\CPinf)_0;\bZ)$ is $-\kappa_{-1,2}$.
\end{prop}
\begin{proof}
From \cite{ERW10}, we know the group $H^2(\mathcal{S}_g(\CPinf)_0;\bZ)$ is torsion-free, and a rational basis is given by $\kappa_1$, $\kappa_{-1,2}$ and $\kappa_{0,1}$. In the fibration (\ref{eq:TopologicalJacobianFibration}) the classes $\kappa_1$ and $\kappa_{-1,2}$ restrict to zero on $\CPinf$, but $\kappa_{0,1}$ restricts to a non-zero class. Thus $\pi^*(\Omega)$ lies in the span of $\kappa_1$ and $\kappa_{-1,2}$. On the other hand, the section $\mathcal{M}_g \to \mathcal{S}_g(\CPinf)_0$ that classifies the trivial line bundle pulls back $\kappa_1$ non-trivially. Thus $\pi^*(\Omega)$ lies in the span of $\kappa_{-1,2}$, and it just remains to determine the constant of proportionality.

There is a commutative diagram
\begin{diagram}
\CPinf & \rTo & \map^0(\Sigma_g, \CPinf) & \rTo & BH\\
\dEq & & \dTo & & \dTo\\
\CPinf & \rTo & \mathcal{S}_g(\CPinf)_0 &\rTo& B\widetilde{\Gamma}_g
\end{diagram}
and the class $\Omega \in H^2(B\widetilde{\Gamma}_g;\bZ)$ pulls back to twice the symplectic form on $BH$. The class $\kappa_{-1,2}$ evaluated on the surface bundle
$$\Sigma_g \lra \Sigma_g \times \map^0(\Sigma_g, \CPinf) \lra \map^0(\Sigma_g, \CPinf)$$
may be computed as follows. The class $c \in H^2(\Sigma_g \times \map^0(\Sigma_g, \CPinf);\bZ)$ is the evaluation map, which in the K\"{u}nneth decomposition is
$$c = \sum a_i \otimes b_i - b_i \otimes a_i + 1 \otimes X$$
for some class $X \in H^2(\map^0(\Sigma_g, \CPinf);\bZ)$. This squares to
$$c^2 = -2[\Sigma_g]^* \otimes \sum a_i \wedge b_i + \cdots$$
which fibre integrates to $-2\sum a_i \wedge b_i \in H^2(\map^0(\Sigma_g, \CPinf);\bZ)$. This is precisely the pullback of minus twice the symplectic form from $BH$. Thus $\pi^*(\Omega) = -\kappa_{-1,2}$.
\end{proof}

To prove Theorem \ref{thm:Main} for general degrees, note that if $c$ is the degree $d$ line bundle on the universal bundle over $\mathcal{S}_g(\CPinf)_d$ then $c' = \frac{1}{\gcd(d, \chi)}(\chi c - d e)$, where $e$ is the Euler class of the vertical tangent bundle, is a degree zero line bundle, and this produces a map
$$f_d : \mathcal{S}_g(\CPinf)_d \lra \mathcal{S}_g(\CPinf)_0.$$
Theorem \ref{thm:Main} now follows from
\begin{prop}
The class $f_d^*(\kappa_{-1,2})$ is
$$\frac{1}{\gcd(d,\chi)^2}\left(\chi^2 \kappa_{-1,2} - 2d\chi \kappa_{0,1} + d^2 \kappa_1 \right) \in H^2(\mathcal{S}_g(\CPinf)_d;\bZ).$$
\end{prop}
\begin{proof}
The class $f_d^*(\kappa_{-1,2})$ is by definition $\pi_!((c')^2)$. We compute
$$(c')^2 = \frac{1}{\gcd(d, \chi)^2}(\chi c - d e)^2 = \frac{1}{\gcd(d, \chi)^2}(\chi^2 c^2 - 2d \chi e c + d^2e^2)$$
and the fibre integral of this class is, by definition of $\kappa_{-1,2}$, $\kappa_{0,1}$ and $\kappa_1$, the element in the statement of the proposition.
\end{proof}

\bibliographystyle{plain}
\bibliography{MainBib}

\def\cprime{$'$}
\begin{thebibliography}{1}

\bibitem{Magma}
Wieb Bosma, John Cannon, and Catherine Playoust.
\newblock The {M}agma algebra system. {I}. {T}he user language.
\newblock {\em J. Symbolic Comput.}, 24(3-4):235--265, 1997.
\newblock Computational algebra and number theory (London, 1993).

\bibitem{ERW10}
Johannes Ebert and Oscar Randal-Williams.
\newblock Stable cohomology of the extended mapping class group and the
  universal {P}icard variety.
\newblock arXiv:1012.0901, 2010.

\bibitem{Faber}
Carel Faber.
\newblock A conjectural description of the tautological ring of the moduli
  space of curves.
\newblock In {\em Moduli of curves and abelian varieties}, Aspects Math., E33,
  pages 109--129. Vieweg, Braunschweig, 1999.

\bibitem{Kawazumi98}
Nariya Kawazumi.
\newblock A generalization of the {M}orita-{M}umford classes to extended
  mapping class groups for surfaces.
\newblock {\em Invent. Math.}, 131(1):137--149, 1998.

\bibitem{Looijenga}
Eduard Looijenga.
\newblock On the tautological ring of $\mathcal{M}_g$.
\newblock {\em Invent. Math.}, 121(2):411--419, 1995.

\bibitem{MoritaJacobianI}
Shigeyuki Morita.
\newblock Families of {J}acobian manifolds and characteristic classes of
  surface bundles. {I}.
\newblock {\em Ann. Inst. Fourier (Grenoble)}, 39(3):777--810, 1989.

\bibitem{MoritaJacobianII}
Shigeyuki Morita.
\newblock Families of {J}acobian manifolds and characteristic classes of
  surface bundles. {II}.
\newblock {\em Math. Proc. Cambridge Philos. Soc.}, 105(1):79--101, 1989.

\bibitem{Mumford}
David Mumford.
\newblock Towards an enumerative geometry of the moduli space of curves.
\newblock In {\em Arithmetic and geometry, {V}ol. {II}}, volume~36 of {\em
  Progr. Math.}, pages 271--328. Birkh\"auser Boston, Boston, MA, 1983.

\end{thebibliography}

\end{document}